\documentclass[11pt]{amsart}

\usepackage{tikz}
\usepackage{verbatim}

\newtheorem{theorem}{Theorem}[section]
\newtheorem{corollary}[theorem]{Corollary}
\newtheorem{lemma}[theorem]{Lemma}
\newtheorem{proposition}[theorem]{Proposition}
\newtheorem{problem}[theorem]{Problem}

\theoremstyle{remark}
\newtheorem{remark}[theorem]{Remark}

\renewcommand{\b}[1]{\mathbf{#1}}

\begin{document}

\title{Partial words with a unique position starting a square}
\author[John Machacek]{John Machacek}
\address{
Department of Mathematics and Statistics\\ York  University\\ Toronto, Ontario M3J 1P3\\ CANADA
}
\email{machacek@yorku.ca}


\begin{abstract}
We consider partial words with a unique position starting a power.
We show that over a $k$ letter alphabet, a partial word with a unique position starting a square can contain at most $k$ squares.
This is in contrast to full words which can contain at most one power if a unique position starts a power.
For certain higher powers we exhibit binary partial words containing three powers all of which start at the same position.
\end{abstract}
\maketitle

\section{Introduction}
We consider how many powers can be contained in a partial if only one position is allowed to start a power.
For full words, it follows from the periodicity theorem of Fine and Wilf~\cite[Theorem 1]{FineWilf} that if only one position starts a power in a full word, then that full word contains only one power.
We give a proof this fact and record it as Corollary~\ref{cor:full}.
This shows we are considering a novel phenomenon of partial words.

Enumerating powers in words and partial words is a fundamental problem in combinatorics on words.
One approach is to count the number of distinct powers.
Extremal words or partial words exhibiting greatest number of distinct powers are of theoretical interest, and they are also of practical importance due to the role repetition plays in string algorithms~\cite{CIR2009,Smyth2013}.
Another approach to enumerating powers is to look at which positions start a power, as it was done for squares in~\cite{positions, ME}.
Here we consider extremal words for the constrained problem of maximizing the number of distinct powers subject to restricting positions starting a power.

The number of distinct squares in a word of length $n$ over any alphabet was shown to be at most $2n$ by Fraenkel and Simpson~\cite{FS1998}.
Another proof of this fact was given by Ilie~\cite{I2005} who also has given an improved bound of $2n- \Theta(\log n)$~\cite{I2007}.
The analogous problem for partial words has also been considered~\cite{BSM2009,BSM,HHK2010, ME}.
A key insight for counting distinct squares in full words is that at most two squares can have their last occurrence starting at a given position regardless of alphabet size.
This is no longer true for partial words.
Halava, Harju, and K\"arki have shown that in a partial word with one hole a maximum of $2k$ squares can have there last occurrence at a given position~\cite[Theorem 2.1]{HHK2010}, and they have given a partial word achieving the bound~\cite[Theorem 2.5]{HHK2010}.

For our problem we also find dependence on the size of the alphabet.
We prove in Theorem~\ref{thm:sq} that a partial word over an alphabet of size $k$ with a unique position starting a square can contain at most $k$ squares.
Furthermore, we show the existence of a partial word meeting this bound.
In Proposition~\ref{prop:2} and Proposition~\ref{prop:3} we construct binary partial words with a unique position starting higher powers.

\section{Preliminaries}
We now give the necessary definitions for our work.
A \emph{full word} is a finite sequence of \emph{letters} taken from a finite set $\Sigma$ called an \emph{alphabet}.
A \emph{partial word} is a finite sequence taken from $\Sigma_{\diamond} := \Sigma \sqcup \{\diamond\}$.
Here $\diamond$ is called a \emph{hole} and represents a ``wild card'' symbol.
Full words are special cases of partial words which contain no holes.
We represent partial words in bold.

If $\b{w}$ is a partial word we denote its length by $|\b{w}|$.
We index positions in our partial words starting from $1$.
We let $\b{w}[i]$ denote the letter or hole in $i$th position of $\b{w}$ and let $\b{w}[i..j]$ denote the \emph{factor} which is the partial word that occupies positions $i$ to $j$ of $\b{w}$.
We may write $\b{w}[..j]$ for the \emph{prefix} $\b{w}[1..j]$ and $\b{w}[i..]$ for the \emph{suffix} $\b{w}[i..|\b{w}|]$.
We define the two sets
\[D(\b{w}) := \{i : 1 \leq i \leq |\b{w}|, \b{w}[i] \neq \diamond \}\]
and
\[H(\b{w}) := \{ i : 1 \leq i \leq |\b{w}|, \b{w}[i] = \diamond \}\]
that keep track which positions of $\b{w}$ are holes.

We say that a partial word $\b{w}$ is \emph{$p$-periodic} if $\b{w}[i] = \b{w}[j]$ for all $i,j \in D(\b{w})$ such that $i \equiv j \pmod{p}$.
This notion of a period in partial words is known as a \emph{strong period}.
We only need the definition of periodicity for full words.
In the case $\b{w}$ is a full word being $p$-periodic just means that $\b{w}[i] = \b{w}[j]$ for all $i \equiv j \pmod{p}$ since $H(\b{w})$ is empty for a full word.
So, it suffices for our purposes to only consider this single definition of a period.
If $\b{v}$ and $\b{w}$ are two partial words with $|\b{v}| = |\b{w}|$ and $D(\b{v}) \subseteq D(\b{w})$ such that  $\b{v}[i] = \b{w}[i]$ for all $i \in D(\b{v})$, then we say $\b{v}$ is \emph{contained} in $\b{w}$ and write $\b{v} \subset \b{w}$.
Partial words $\b{u}$ and $\b{v}$ are said to be \emph{compatible} if there exists a third partial word $\b{w}$ such that $\b{u} \subset \b{w}$ and $\b{v} \subset \b{w}$.
In the case $\b{u}$ and $\b{v}$ are compatible we write $\b{u} \uparrow \b{v}$.

We now give the definition we are primarily concerned with.
An \emph{$r$th power} is a partial word $\b{w}$ such that $\b{w} \subset \b{x}^r$ for some full word $\b{x}$.
Often $2$nd powers are called \emph{squares} while $3$rd powers are commonly referred to as \emph{cubes}.

To review the definitions in this section consider as an example the partial word
\[\b{w} = a\diamond b\, a\,c \diamond a\,c\,b\]
where the alphabet is $\Sigma = \{a,b,c\}$.
Here $|\b{w}| = 9$ while $H(\b{w}) = \{2,6\}$.
The factor $\b{w}[1..3] = a \diamond b$ is contained in $a\,c\,b$.
In fact all the factors $\b{w}[3i-2..3i]$ for $1 \leq i \leq 3$ are all contained in $a\,c\,b$.
This makes $\b{w} \subset (a\,c\,b)^3$ a cube.

\section{Unique power position full words}
Let us now show that any full word with multiple powers starting at the same position must contain a power starting at another position.
We first recall the periodicity theorem of Fine and Wilf.

\begin{theorem}[{\cite[Theorem 1]{FineWilf}}]
Let $\b{w}$ be a full word with periods $p$ and $q$.
If $|\b{w}| \geq p + q - \gcd(p,q)$, then $\b{w}$ has period $\gcd(p,q)$.
\label{thm:FW}
\end{theorem}

We can now prove a corollary of this theorem.
This corollary solves our problem of maximizing the number of powers subject to constraint that a unique position may start a power in the case of full words.

\begin{corollary}
If $\b{w}$ is a full word with two $r$th powers starting at position $i$, then there exists a position $j > i$ which starts an $r$th power.
\label{cor:full}
\end{corollary}
\begin{proof}
It suffices to consider a full word $\b{w}$ such that $\b{w}$ is an $r$th power and a proper prefix of it is also an $r$th power.
Hence $\b{w} = \b{x}^r$ for some $\b{x}$ with $|\b{x}| = p > 1$ and $\b{w}[1..rk] = \b{y}^r$ for some $\b{y}$ with $0 < |\b{y}| = q < p$.
We must show that some $r$th power starts at position $j > 1$.
If $rq \leq p$, then position $p+1$ starts an $r$th power equal to $\b{y}^r$.
So, we may assume that $rq > p$.

Furthermore, if $r=2$ and $2q > p$, then $\b{y} = \b{uv}$ where $\b{x} = \b{yu}$ and $|\b{u}| > 0$.
Thus, $\b{x}^2 = \b{uv}\b{u}^2\b{vu}$ and there is a square starting at a position $j > 1$.
Hence, we may also assume $r > 2$.

If $q$ divides $p$ and $rq > p$, then it follows that $p = r'q$ for some $0 < r' < r$.
This means $\b{x} = \b{y}^{r'}$ and so $\b{w} = \b{y}^{rr'}$.
In particular, this means $j = q+1$ starts an $r$th power equal to $\b{y}^r$.
So, we may also assume $q$ does not divide $p$.

In the case $rq \geq p + q - \gcd(p,q)$, by Theorem~\ref{thm:FW} it follows that $\b{w}[1..rq]$ is $\gcd(p,q)$-periodic.
Since $q$ does not divide $p$ we have $\gcd(p,q) < q$ and $j = 2$ starts an $r$th power.

In the only remaining case we have $p < rq < p + q - \gcd(p,q)$.
It follows that $|\b{y}^r| = rq < 2p = |\b{x}^2|$.
Thus we see that $\b{x}^2$ contains the $r$th power $\b{y}^r$.
Recalling that we are assuming $r > 2$ we conclude that position $j=p+1$ starts an $r$th power equal to $\b{y}^r$.
\end{proof}

\section{Unique square position partial words}

We now turn our attention to partial words having established that full words cannot contain multiple powers if a single position is allowed to start a power.
Let us now characterize how many squares a partial word may contain if only one position starts a square.
We first prove some results which we use to deduce our main theorem on squares.

\begin{lemma}
If $\b{w}$ is a partial word which contains more than one square but only has one position that starts a square, then $H(\b{w}) = \{1\}$.
\label{lem:H1}
\end{lemma}
\begin{proof}
If $\b{w}[i] = \diamond$ for $1 < i < |\b{w}|$, then $\b{w}[i-1..i]$ and $\b{w}[i..i+1]$ are both squares starting at different positions.
If $\b{w}[n] = \diamond$ where $n = |\b{w}|$, then position $n-1$ starts a square.
However, position $n-1$ can only start the single square $\b{w}[n-1..n]$.
By Corollary~\ref{cor:full} no full word can have the property that it contains more than one square but only has a single position which starts a square.
Hence, $H(\b{w}) \neq \emptyset$ and it follows that $H(\b{w}) = \{1\}$.
\end{proof}

\begin{remark}
We now see that if a single position in a partial word starts a square and multiple squares start at this position, then the unique position starting the square must be the first position.
This can be helpful to keep in mind.
We typically say something like ``partial words with a unique position starting a square'' rather than ``partial words that have all squares starting on the first position'' since the former is more general and implies the latter when multiple squares are present.
\end{remark}

\begin{figure}
\begin{tikzpicture}
\draw (-4,2)--(4,2);
\draw (-4,1.5)--(4,1.5);
\draw (-4,2)--(-4,1.5);
\draw (0,2)--(0,1.5);
\draw (4,2)--(4,1.5);

\draw (-4,1.5)--(-4,1);
\draw (-1.5,1.5)--(-1.5,1);
\draw (1,1.5)--(1,1);
\draw (-4,1)--(1,1);
\draw[dotted] (-3,1.5)--(-3,1);
\draw[dotted] (-2.5,1.5)--(-2.5,1);
\draw[dotted] (-0.5,1.5)--(-0.5,1);
\draw[dotted] (0,1.5)--(0,1);

\draw (0,1)--(0,0.5);
\draw (2.5,1)--(2.5,0.5);
\draw (0,1)--(2.5,1);
\draw (0,0.5)--(2.5,0.5);
\draw[dotted] (1,1)--(1,0.5);
\draw[dotted] (1.5,1)--(1.5,0.5);

\node at (-2, 1.75) {$\b{u}$};
\node at (2, 1.75) {$\b{v}$};
\node at (-3.5,1.25) {$\b{x}$};
\node at (-2.75,1.25) {$\b{y}$};
\node at (-2,1.25) {$\b{x'}$};
\node at (-1,1.25) {$\b{x''}$};
\node at (-.25,1.25) {$\b{y}$};
\node at (0.5,1.25) {$\b{x'}$};
\node at (0.5,0.75) {$\b{x'}$};
\node at (1.25,0.75) {$\b{y}$};
\node at (1.75,0.75) {$\b{x'}$};

\draw (-4,0)--(4,0);
\draw (-4,0)--(-4,-0.5);
\draw (0,0)--(0,-0.5);
\draw (4,0)--(4,-0.5);
\draw (-4,-0.5)--(4,-0.5);
\draw (-4,-0.5)--(-4,-1);
\draw (-0.5,-0.5)--(-0.5,-1);
\draw (3,-0.5)--(3,-1);
\draw (-4,-1)--(3,-1);
\draw[dotted] (-3.5,-0.5)--(-3.5,-1);
\draw[dotted] (-3,-0.5)--(-3,-1);
\draw[dotted] (0,-0.5)--(0,-1);
\draw[dotted] (0.5,-0.5)--(0.5,-1);
\draw[dotted] (1,-0.5)--(1,-1);

\node at (-2,-0.25) {$\b{u}$};
\node at (2,-0.25) {$\b{v}$};
\node at (-3.75,-0.75) {$\b{x}$};
\node at (-3.25,-0.75) {$\b{x'}$};
\node at (-1.5,-0.75) {$\b{y}$};
\node at (-0.25,-0.75) {$\b{x''}$};
\node at (0.25,-0.75) {$\b{x'}$};
\node at (0.75,-0.75) {$\b{x'}$};
\end{tikzpicture}
\caption{Partial words in proof of Lemma~\ref{lem:2k}.}
\label{fig:1}
\end{figure}
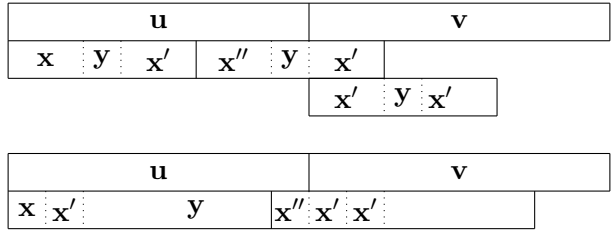

\begin{lemma}
Let $\b{u} = \diamond \b{u'}$ for a full word $\b{u'}$.
Let $\b{v}$ be any full word compatible with $\b{u}$, and set $\b{w} = \b{uv}$.
If $\b{w}[1..2k]$ is a square and $|\b{u}| < 2k < |\b{w}|$, then there exists a position $1 < i < |\b{w}|$ which starts a square in $\b{w}$.
\label{lem:2k}
\end{lemma}

\begin{proof}
A diagram depicting the partial words in this proof can be found in Figure~\ref{fig:1}.
Assume that $\b{w}[1..2k]$ is a square and $|\b{u}| < 2k < |\b{w}|$.
If $3k/2 \leq |\b{u}| < 2k$, then we must have
\[\b{w}[1..(|\b{u}| + k)] = \b{xyx'x''yx'yx'}\]
where $\b{x'}$ and $\b{x''}$ are nonempty full words compatible with the partial word $\b{x}$ where $H(\b{x}) = \{1\}$.
Here $|\b{x}| = |\b{x'}| = |\b{x''}| = 2k - |\b{u}|$ while $|\b{y}| = k - 2|\b{x}|$.
Under the assumption $3k/2 \leq |\b{u}| < 2k$ it follows that $0 < |\b{x}| \leq k/2$.
In particular, $\b{x}$, $\b{x'}$, and $\b{x''}$ are all nonempty.
Thus, $\b{w}$ contains the square $\b{yx'yx'}$ starting at a position greater than $1$.

If $|\b{u}| \leq 3k/2$, then 
\[\b{w}[1..(3|\b{u}| - 2k)] = \b{xx'yx''x'x'}\]
where $\b{x'}$ and $\b{x''}$ are nonempty full words compatible with the partial word $\b{x}$ where $H(\b{x}) = \{1\}$.
Here $|\b{x}| = |\b{x'}| = |\b{x''}| = |\b{u}| - k$ while $|\b{y}| = k - 2|\b{x}|$.
Since $2k < |\b{w}|$ we have $k < |\b{u}|$.
Under the assumption $|\b{u}| \leq 3k/2$ we find that $0 < |\b{x}| \leq k/2$.
Thus,  $\b{x}$, $\b{x'}$, and $\b{x''}$ are all nonempty.
In this case $\b{w}$ contains the square $\b{x'x'}$ starting at a position greater than $1$ and the proof is complete.
\end{proof}

\begin{lemma}
Let $\b{u} = \diamond \b{u'}$ for a full word $\b{u'}$.
Let $\b{v}$ be any full word compatible with $\b{u}$, and set $\b{w} = \b{uv}$.
If $\b{w}[1..2k]$ is a square with $2k \leq |\b{u}|$ such that $\b{w}[k+1] = \b{v}[1]$, then $\b{v}$ contains a square.
Hence, $\b{w}$ contains a square starting at position $1 < i < |\b{w}|$.
\label{lem:short}
\end{lemma}
\begin{proof}
Let $\b{x} = \b{w}[1..k]$ and $\b{y} = \b{w}[k+1..2k]$.
Then $H(\b{x}) = \{1\}$, $\b{y}$ is a full word, and $\b{x}[2..k]=\b{y}[2..k]$.
If $\b{v}[1]=\b{w}[k+1]=\b{y}[1]$, then $\b{v}[1..k] = \b{y}$.
Therefore, $\b{v}[1..2k] = \b{y}^2$ and $\b{v}$ contains a square.
\end{proof}

\begin{theorem}
Any partial word over an alphabet of size $k$ with a unique position starting a square contains at most $k$ squares.
Furthermore, for any $k$ there exists a partial word over an alphabet of size $k$ with a unique position starting a square containing exactly $k$ squares.
\label{thm:sq}
\end{theorem}

\begin{proof}
Consider a partial work over an alphabet $\Sigma$ of size $k$ with a unique position starting a square.
Let $\b{w} = \b{w}[1..2n]$ be a longest square.
By Lemma~\ref{lem:H1} we must have $H(\b{w}) = \{1\}$.
Assume the lengths of the squares starting at position $1$ in $\b{w}$ are $2\ell_1 < 2\ell_2 < \cdots < 2\ell_{m} = 2n$.
Then by Lemma~\ref{lem:2k} it follows that $2\ell_i \leq \ell_{i+1}$.
By Lemma~\ref{lem:short} it follows that $\b{w}[\ell_i+1] \neq \b{w}[\ell_j +1]$ for any $i < j$.
In particular, $\b{w}[\ell_1 + 1], \b{w}[\ell_2 + 1], \dots, \b{w}[\ell_m +1]$ must all be distinct letters of $\Sigma$.
However, since $|\Sigma| = k$  it must be that $m \leq k$.
Therefore such a partial word $\b{w}$ can contain at most $k$ squares.

We now construct a partial word meeting the bound we have just proven.
Given an alphabet $\Sigma = \{a_1, a_2, \dots, a_k\}$ start with $\b{w}_0 = \diamond$.
Then put $\b{w}_{i+1} = \b{w}_{i} a_{i+1} \b{w}_{i}[2..]$ for $1 \leq i \leq k$.
By construction we see that $|\b{w}_i| = 2^i$ for each $1 \leq i \leq k$ and that $\b{w}_k[1..2^j]$ is a square for each $1 \leq j \leq k$.
It only remains to show that $\b{w}_k[i..j]$ is not a square for any $i \neq 1$.

We show by induction that $\b{w}_i$ does not have any squares which do not start at position $1$.
The base case for $w_0$ is clear.
Now assume the claim for $\b{w}_i$ and consider $\b{w}_{i+1}$.
Take the factor $\b{w}_{i+1}[j_1 .. j_2]$ for some $j_1 \neq 1$.
If $j_2 < 2^i$, then this is a factor of $\b{w}_i$ and thus not a square by induction.
If $j_2 \geq 2^i$, then the factor is a full word containing exactly one occurrence of the letter $a_{i+1}$ and hence not a square. 
Therefore $\b{w}_k$ is a partial word containing $k$ squares all of which start at position $1$.
\end{proof}

Examples of partial words constructed in the proof of Theorem~\ref{thm:sq} are
\begin{align*}
\b{w}_0 &= \diamond \\
\b{w}_1 &=\diamond \, a\\
\b{w}_2 &= \diamond \, a\,b\,a \\
\b{w}_3 &= \diamond \, a\,b\,a\,c\, a\,b\,a\\
\b{w}_4 &= \diamond \, a\,b\,a\,c\, a\,b\,a\,d \, a\,b\,a\,c\, a\,b\,a
\end{align*}
with the alphabet $\Sigma = \{a,b,c,d\}$.

\section{Conclusion and future work}

In this section we begin the investigation of partial words with a unique position starting an $r$th power for $r > 2$.
Unlike for the $r=2$ case we do not completely solve the problem for $r > 2$.
We first show that we can always find a binary partial word with two $r$th powers starting at position $1$ such that position $1$ is the only position which starts an $r$th power.

\begin{proposition}
If $r \geq 2$, then the partial word
\[w = \diamond^{r-1} a b a \diamond^{r-2}\]
contains exactly two $r$th powers both of which start at position $1$.
\label{prop:2}
\end{proposition}

\begin{proof}
We first see that $\b{w}[1..r]$ is an $r$th power contained in $a^r$.
Also, $\b{w}[1..2r]$ is an $r$th power contained in $(ba)^r$ if $r$ is even or otherwise contained in $(ab)^r$ if $r$ is odd.
For any $i > 1$, $\b{w}[i..i+r]$ cannot be an $r$th power since it always contains both the letters $a$ and $b$.
The proposition is now proven since $\b{w}$ is too short to contain any other $r$th powers of length $2r$ or greater.
\end{proof}

For $r>2$ the situation is more complicated than for squares.
Partial words with a unique position starting an $r$th power can contain more than a single hole when $r>2$.
Also in contrast to the case of squares, there are partial words over an alphabet of size $k$ containing more than $k$ cubes.
Over $\Sigma = \{a,b\}$ the partial words
\begin{align*}
\b{v} &= \diamond\diamond a\,b\,a \diamond b\,a\,a\\
\b{w} &= \diamond \diamond a\,b\,a \diamond b\,a\, \diamond
\end{align*}
both contain three cubes all of which start at the first position.
Let us now give a generalization of the construction of $\b{w} = \diamond \diamond a\,b\,a \diamond b\,a\ a$.

\begin{proposition}
Assume that $r$ is odd and $r \equiv 0 \pmod{3}$. Then the partial word
\[w = \diamond^{r-1}aba\diamond^{r-2} b a^2 \diamond^{r-3}\]
contains exactly three $r$th powers all of which start at position $1$.
\label{prop:3}
\end{proposition}

\begin{proof}
We first note that $\b{w}[1..r]$ is an $r$th power contained in $a^r$.
Next we see that $\b{w}[1..2r]$ is an $r$th power $(ab)^r$ since $r$ is odd.
The third $r$th power starting at position $1$ is $\b{w}=\b{w}[1..3r]$ which is contained in $(ba^2)^r$ where we have used that $r \equiv 0 \pmod{3}$.

Now $\b{w}$ can contain no other $r$th power of length $r$ since $\b{w}[i..i+r-1]$ always contains both the letters $a$ and $b$ for $i > 1$.
Any other $r$th power of length $2r$ must be $\b{v}=\b{w}[i..i+2r-1]$ where $i > 1$.
So, $b = \b{w}[r+1] = \b{v}[r+2-i]$ and $b = \b{w}[2r+1] = \b{v}[2r+2-i]$.
Since $r$ is odd we see that $b$ occurs in positions with opposite parity in $\b{v}$.
It follows that $\b{v}$ cannot be an $r$th power since $\b{v}$ also contains an $a$.
Therefore the proof is complete since $\b{w}$ is too short to contain any other $r$th powers of length $3r$ or greater.
\end{proof}

We let $M(r,k)$ denote the maximum possible number of $r$th powers contained in any partial word over an alphabet of size $k$ with a unique position starting an $r$th power.
In Theorem~\ref{thm:sq} we have shown that $M(2,k) = k$.
It is clear that $M(r,k) \leq M(r, k+1)$.
Proposition~\ref{prop:2} shows that $M(r,2) \geq 2$ while Proposition~\ref{prop:3} shows that $M(3(2s+1),2) \geq 3$.
We pose the following problem.

\begin{problem}
Compute $M(r,k)$ for $r > 2$ and $k > 1$.
\end{problem}

This problem appears more difficult for higher powers than it is for squares.
For higher powers multiple holes can be present, but the partial word can still contain a unique position which starts a power.
One can further generalize the problem and look for the maximum possible number of $r$th powers contained in a word over a $k$ letter alphabet where at most $t$ positions are allowed to begin $r$th powers.
The problem we considered in this article is recovered when $t = 1$.

\bibliographystyle{alpha}
\bibliography{refs}

\end{document}